\documentclass[11 pt]{article}
\usepackage[utf8]{inputenc}
\usepackage{amsmath}
\usepackage{amsfonts}
\usepackage{amssymb}
\usepackage{graphicx}
\usepackage{mathrsfs}
\usepackage{upref,amsthm,amsxtra,exscale}
\usepackage{cite}
\usepackage[colorlinks=true,urlcolor=blue,
citecolor=red,linkcolor=blue,linktocpage,pdfpagelabels,
bookmarksnumbered,bookmarksopen]{hyperref}
\usepackage{cleveref}
\usepackage{fullpage}

\usepackage{subcaption}
\usepackage{caption}
\newtheorem{theorem}{Theorem}[section]

\newtheorem{remark}[theorem]{Remark}
\newtheorem{lemma}[theorem]{Lemma}

\newtheorem{definitions}[theorem]{Definitions}
\newtheorem{examples}[theorem]{Examples}
\numberwithin{equation}{section}

\newcommand{\R}{\mathbb R}

\newcommand{\cH}{{\mathcal H}}
 \def\cN{\mathcal{N}}
 \def\eps{\varepsilon}

\def\r{\mathbb{R}}
\def\rn{\mathbb{R}^N}
\def\z{\mathbb{Z}}

\def\n{\mathbb{N}}
\def\cc{\mathbb{C}}
\def\eps{\varepsilon}

\def\io{\int_{\Omega}}
\def\irn{\displaystyle\int_{\r^N}}

\def\o{\Omega}

\def\bf{\boldsymbol}

\def\cC{\mathcal{C}}

\def\cH{\mathcal{H}}

\def\cJ{\mathcal{J}}

\def\cN{\mathcal{N}}

\def\cU{\mathcal{U}}

\def\bar{\overline}
\def\what{\widehat}
\def\e{\mathrm{e}}
\def\d{\,\mathrm{d}}
\def\backslash{\smallsetminus}
\def\dsum{\displaystyle\sum}

\usepackage[dvipsnames]{xcolor}
\usepackage{marvosym}

\title{Exponential decay of the solutions to nonlinear Schrödinger systems}
\author{Felipe Angeles\footnote{Instituto de Matemáticas, Universidad Nacional Autónoma de México, Circuito Exterior, Ciudad Universitaria, 04510 Coyoacán, Ciudad de México, Mexico, 
\texttt{felidaujal@im.unam.mx}}\,,\ M\'onica Clapp\footnote{Instituto de Matemáticas, 
Universidad Nacional Autónoma de México, Campus Juriquilla, Boulevard Juriquilla 3001, 76230 Querétaro, Qro., Mexico, \texttt{monica.clapp@im.unam.mx} }\,, and Alberto Salda\~na (\Letter)\footnote{(Corresponding author \Letter) Instituto de Matemáticas, Universidad Nacional Autónoma de México, Circuito Exterior, Ciudad Universitaria, 04510 Coyoacán, Ciudad de México, Mexico, \texttt{alberto.saldana@im.unam.mx} }}
\date{}

\begin{document}

\maketitle

\begin{abstract}
We show that the components of finite energy solutions to general nonlinear Schrödinger systems have exponential decay at infinity. Our results apply to positive or sign-changing components, and to cooperative, competitive, or mixed-interaction systems. As an application, we use the exponential decay to derive an upper bound for the least possible energy of a solution with a prescribed number of positive and nonradial sign-changing components. 
\medskip

\medskip


\end{abstract}

\section{Introduction}

Consider the nonlinear Schrödinger system
\begin{equation} \label{eq:introsystem}
\begin{cases}
-\Delta u_i+ V_{i}(x) u_i = \dsum_{j=1}^\ell \beta_{ij}|u_j|^p|u_i|^{p-2}u_i, \\
u_i\in H^1(\rn),\qquad i=1,\ldots,\ell,
\end{cases}
\end{equation}
where $N\geq 1$, $V_i\in L^\infty(\rn)$, $\beta_{ij}\in\mathbb{R}$ and $1<p<\frac{2^*}{2}$. Here $2^*$ is the usual critical Sobolev exponent, namely, $2^{\ast}:=\frac{2N}{N-2}$ if $N\geq 3$ and $2^*:=\infty$ for $N=1,2$. 

Systems of this type occur as models for various natural phenomena. In physics, for example, they describe the behavior of standing waves for a mixture of Bose-Einstein condensates of different hyperfine states which overlap in space \cite{EGBB97}. The coefficients $\beta_{ij}$ determine the type of interaction between the states; if $\beta_{ij}>0$, then there is an attractive force between $u_i$ and $u_j$, similarly, if $\beta_{ij}<0$, then the force is repulsive, and if $\beta_{ij}=0$, then there is no direct interaction between these components.  Whenever all the interaction coefficients are positive, we say that the system is cooperative. If $\beta_{ii}>0$ and $\beta_{ij}<0$ for all $i\neq j$, then the system is called competitive.  And if some $\beta_{ij}$ are positive and others are negative for $i\neq j$, then we say that the system has mixed couplings.  All these regimes exhibit very different qualitative behaviors and have been studied extensively in recent years, see for instance \cite{clapp2022,CPV22,PW13,B16,cp,cp2,cs, DP21,sw15,SW152,S15,ST16,TY20,TYZ22,WW20} and the references therein.

System \eqref{eq:introsystem} has a variational structure, and therefore a natural strategy is to find weak solutions by minimizing an associated energy functional on a suitable set, under additional assumptions on the matrix $(\beta_{ij})$ and on the potentials $V_i$. Using this approach, several kinds of solutions have been found in terms of their signs and their symmetries.  However, there seems to be no information available about the decay of these solutions at infinity. In this paper, we show that finite energy solutions must decay exponentially at infinity, and a rate can be found in terms of the potentials $V_i$. Our main result is the following one.

\begin{theorem}
\label{mainresult}
Assume that, for every $i=1,\ldots,\ell$,
\begin{itemize}
\item[$(V_1)$] $V_i:\rn\to\r$ is Hölder continuous and bounded,
\item[$(V_2)$] there exists $\rho\geq 0$ such that
$$\sigma_i:=\inf_{\R^N\backslash B_\rho(0)} V_i>0.$$
\end{itemize}
Let $(u_1,\ldots,u_\ell)\in \left(H^{1}(\mathbb{R}^{N})\right)^{\ell}$ be a solution of \eqref{eq:introsystem} and let $\mu_i\in(0,\sqrt{\sigma_i})$. Then, there is $C>0$ such that
\begin{align}
\label{rabierexpdecay}
|u_{i}(x)|\leq C \e^{-\mu_i|x|}\qquad \text{ for all }x\in\R^N\text{ and }i=1,\ldots,\ell.
\end{align}
Furthermore, if $V_i\equiv 1$ for every $i=1,\ldots,\ell$, then \eqref{rabierexpdecay} holds true with $\mu_i=1$.
\end{theorem}

We emphasize that each component may have a different decay depending on each potential $V_i$.   The main obstacle to showing \eqref{rabierexpdecay} is to handle the possibly sublinear term $|u_i|^{p-2}u_i$ for $p\in(1,2)$ (which is always the case for $N\geq 4$). To explain this point in more detail, assume that $(u_1,\ldots,u_\ell)$ is a solution of \eqref{eq:introsystem} and write the $i$-th equation of the system as
\begin{align}\label{se}
-\Delta u_i+\big(a_i(x) - c_i(x)|u_i(x)|^{p-2}\big)u_i=0, \qquad 
a_i:=V_i-\beta_{ii}|u_i|^{2p-2},\qquad 
c_i:=\dsum_{j\neq i}^\ell\beta_{ij}|u_j|^p.
\end{align}
Since every $u_j\in H^1(\R^N)\cap\cC^0(\rn)$, we know that $a_i$ and $c_i$ are bounded in $\rn$, but $|u_i|^{p-2}\to \infty$ as $|x|\to\infty$ and it is also singular at the nodal set of a sign-changing solution.  As a consequence, one cannot use directly previously known results about exponential decay for scalar equations, such as those in \cite{ackermann2016,bs,RS00}. In fact, one can easily construct a one dimensional solution of a similar scalar equation that has a power-type decay. For instance, let $w\in \cC^2(\R)$ be a positive function such that $w(x)=|x|^{-2/3}$ for $|x|>1$ and let
\begin{align*}
    c(x):=\frac{-w''(x)+w(x)}{w(x)^\frac{1}{2}},\qquad x\in\R.
\end{align*}
Then, $w\in H^1(\R)$ is a solution of $-w''+w=c\,w^\frac{1}{2}$ in $\R$, $c(x)\to 0$ as $|x|\to \infty$, and $w$ decays as a power at infinity.

This shows that the proof of the exponential estimate in Theorem \ref{mainresult} must rely on a careful study of the \emph{system structure}. In other words, although the sublinear nonlinearity $|u_i|^{p-2}u_i$ appears in \eqref{eq:introsystem}, the system is not sublinear. As a whole, it is always superlinear.

With this in mind, we adapt some of the arguments in \cite{ackermann2016,RS00} preserving at each step the system structure of the problem.  These arguments rely basically on elliptic regularity and comparison principles.

The exponential decay of solutions is a powerful tool in their qualitative study. As an application of Theorem \ref{mainresult}, we derive \emph{energy bounds} of solutions having prescribed positive and nonradial sign-changing components. For this, power type decay would not be enough. 

To be more precise, we consider the autonomous system
\begin{equation} \label{eq:system}
\begin{cases}
-\Delta u_i+ u_i = \dsum_{j=1}^\ell \beta_{ij}|u_j|^p|u_i|^{p-2}u_i, \\
u_i\in H^1(\rn),\qquad i=1,\ldots,\ell.
\end{cases}
\end{equation}
where the $\beta_{ij}$'s satisfy the following condition:
\begin{itemize}
\item[$(B_1)$] The matrix $(\beta_{ij})$ is symmetric and admits a block decomposition as follows: For some $1 \leq q \leq\ell$ there exist \ $0=\ell_0<\ell_1<\dots<\ell_{q-1}<\ell_q=\ell$ \ such that, if we set
\begin{align*}
        I_h:= \{i \in  \{1,\dots,\ell\}:  \ell_{h-1} < i \le \ell_h \},\qquad h\in\{1,\ldots,q\},
\end{align*}
then $\beta_{ii}>0, \ \beta_{ij}\geq 0 \text{ if } i,j\in I_h$, \ and \ $\beta_{ij}<0 \text{ if } i\in I_h, \ j\in I_k\text{ and }h\neq k$.
\end{itemize}

According to this decomposition, a solution $\bf u=(u_1,\ldots,u_\ell)$ to \eqref{eq:introsystem} may be written in block-form as
\[\bf u=(\bar u_1,\ldots,\bar u_q)\qquad\text{with \ }\bar u_h=(u_{\ell_{h-1}+1},\ldots,u_{\ell_h}),\quad h=1,\ldots,q.\]
We say that $\bf u$ is \emph{fully nontrivial} if every component $u_i$ is different from zero. 

Set $Q:=\{1,\ldots,q\}$. Given a partition $Q=Q^+\cup Q^-$ with $Q^+\cap Q^-=\emptyset$ we look for solutions such that every component of $\bar u_h$ is positive if $h\in Q^+$ and every component of $\bar u_h$ is nonradial and changes sign if $h\in Q^-$. To this end, we use variational methods in a space having suitable symmetries. As shown in \cite[Section 3]{clapp2022}, to guarantee that the solutions obtained are fully nontrivial we need to assume the following two conditions:

\begin{itemize}
\item[$(B_2)$] For each $h\in Q$, the graph whose set of vertices is $I_h$ and whose set of edges is $E_h:=\{\{i,j\}:i,j\in I_h, \ i\neq j, \ \beta_{ij}>0\}$ is connected.
\item[$(B_3)$] If $q\geq 2$ then, for every $h\in\{1,\ldots,q\}$ such that $\ell_{h}-\ell_{h-1}\geq 2$, the inequality
$$\Big(\min_{\{i,j\}\in E_h}\beta_{ij}\Big)\left[\frac{\min\limits_{h=1,\ldots,q} \ \max\limits_{i\in I_h}\beta_{ii}}{\dsum_{i,j\in I_h}\beta_{ij}}\right]^\frac{p}{p-1}>\,C_*\dsum_{\substack{k=1 \\ k\neq h}}^q \ \dsum_{\substack{i\in I_h \\ j\in I_k}}|\beta_{ij}|$$
holds true, where $C_*=C_*(N,p,q,Q^+)>0$ is the explicit constant given in \eqref{C} below. 
\end{itemize}

In \cite{clapp2022} it is shown that, for any $q$, the system \eqref{eq:introsystem} has a fully nontrivial solution satisfying the sign requirements described above. Furthermore, an upper bound for its energy is exhibited, but only for systems with at most $2$ blocks, i.e., for $q=1,2$. Here we use Theorem \ref{mainresult} to obtain an energy bound for \emph{any} number of blocks. 

For each $h=1,\ldots,q$, let \ $\r^{I_h}:=\{\bar s=(s_{\ell_{h-1}+1},\ldots,s_{\ell_h}):s_i\in\r\text{ for all }i\in I_h\}$ \ and define
\begin{equation} \label{eq:mu_h}
\mu_h:=\inf_{\substack{\bar s\in\r^{I_h} \smallskip \\ \bar s\neq 0}}\left(\frac{\sum_{i\in I_h}s_i^2}{\Big(\sum_{{i,j\in I_h}}\beta_{ij}|s_i|^p|s_j|^p\Big)^\frac{2}{2p}}\right)^\frac{p}{p-1}.
\end{equation}
For any $\ell\in\n$, we write $\|\bf u\|$ for the usual norm of $\bf u=(u_1,\ldots,u_\ell)$ in $(H^1(\rn))^\ell$, i.e.,
$$\|\bf u\|^2:=\sum_{i=1}^\ell\irn(|\nabla u_i|^2+|u_i|^2).$$
We prove the following result.

\begin{theorem} \label{thm:main2}
Let $N=4$ or $N\geq 6$, and let $Q=Q^+\cup Q^-$ with $Q^+\cap Q^-=\emptyset$. Assume $(B_1)$, $(B_2)$, and $(B_3)$. Then, there exists a fully nontrivial solution $\bf u=(\bar u_1,\ldots,\bar u_q)$ to the system \eqref{eq:system} with the following properties:
\begin{itemize}
\item[$(a)$] Every component of $\bar u_h$ is positive if $h\in Q^+$ and every component of $\bar u_h$ is nonradial and changes sign if $h\in Q^-$.
\item[$(b)$] If $q=1$, then
$$\|u\|^2=\mu_1\|\omega\|^2\text{ \ if \ }Q=Q^+\qquad\text{and}\qquad\|u\|^2<10\,\mu_1\|\omega\|^2\text{ \ if \ }Q=Q^-.$$
\item[$(c)$] If $q\geq 2$ the following estimate holds true
\begin{equation}\label{ee}
\|\bf u\|^2< \left(\min_{k\in Q}\Big(a_k\mu_k+\sum_{h\in Q\smallsetminus\{k\}}b_h\mu_h\Big)\right)\|\omega\|^2,
\end{equation}
\end{itemize}
where $a_k:=1$ if $k\in Q^+$, $a_k:=12$ if $k\in Q^-$, $b_h:=6$ if $h\in Q^+$, $b_h:=12$ if $h\in Q^-$, and $\omega$ is the unique positive radial solution to the equation
\begin{equation} \label{eq:omega}
-\Delta w+w=|w|^{2p-2}w,\qquad w\in H^1(\rn).
\end{equation} 
\end{theorem}

To prove Theorem \ref{thm:main2}, we follow the approach in \cite{clapp2022} and impose on the variational setting some carefully constructed symmetries which admit finite orbits. This approach immediately gives energy estimates but it requires showing a quantitative compactness condition which needs precise knowledge about the asymptotic decay of the components of the system. Here is where we use Theorem \ref{mainresult}.

The paper is organized as follows. Section \ref{exp:sec} is devoted to the proof of the exponential decay stated in Theorem \ref{mainresult}. The application of this result to derive energy bounds is contained in Section \ref{ee:sec}, where we also give some concrete examples. 

\subsection*{Acknowledgments}

We thank Nils Ackermann for helpful comments and suggestions.  F. Angeles and A. Saldaña thank the Instituto de Matemáticas - Campus Juriquilla for the kind hospitality. F. Angeles is supported by CONACYT (Mexico) through a postdoctoral fellowship  under grant A1-S-10457. M. Clapp is supported by CONACYT (Mexico) through the research grant A1-S-10457. A. Saldaña is supported by UNAM-DGAPA-PAPIIT (Mexico) grant IA100923 and by CONACYT (Mexico) grant A1-S-10457.

\section{Exponential decay}\label{exp:sec}

This section is devoted to the proof of Theorem \ref{mainresult}. As a first step, we extend the argument in \cite[Lemma 5.3]{ackermann2005} to systems. Let $B_r$ denote the ball of radius $r$ in $\rn$ centered at zero. Let $\sigma_i$ and $\beta_{ij}$ as in $(V_2)$ and \eqref{eq:introsystem}, then we let $\boldsymbol{\sigma}:=(\sigma_1,\ldots,\sigma_\ell)$ and $\boldsymbol{\beta}:=(\beta_{ij})_{i,j=1}^\ell$.

\begin{lemma}
\label{normdecay}
Let $V_i\in L^\infty(\R^N)$ satisfy $(V_2)$ and let $\boldsymbol{u}=(u_1,\ldots,u_\ell)$ be a solution of \eqref{eq:introsystem}. Set
\begin{align*}
\xi_{i}(r):=\int_{\mathbb{R}^{N}\smallsetminus B_{r}} \big(|\nabla u_{i}|^{2}+|u_{i}|^{2}\big) \qquad \text{ and } \qquad \bf\xi(r):=(\xi_1(r),\ldots,\xi_\ell(r)).
\end{align*}
Then, there are positive constants $C=C(\boldsymbol u, \boldsymbol\sigma,\boldsymbol\beta,N,\rho,p)$ and $\vartheta=\vartheta(\boldsymbol\sigma)$, with $\rho$ and $\sigma_i$ as in $(V_2)$, such that 
$$|\bf\xi(r)|_1:=\sum_{i=1}^\ell\xi_i(r)\leq C \e^{-\vartheta r}\qquad\text{for every \ }r\geq 0.$$
\end{lemma}

\begin{proof}
Let $\chi:\R^N\to\R$ be given by $\chi(r):=0$ if $r\leq 0$, \ $\chi(r):=r$ if $r\in(0,1)$ \ and \ $\chi(r):=1$ if $r\geq 1$. Let $u_{i}^{r}(x):=\chi(|x|-r)u_{i}(x)$ for $r\geq 0$, $x\in\R^N$, and $i=1,\ldots,\ell$. Then $u_{i}^{r}\in H^{1}(\mathbb{R}^{N})$ and
\begin{align*}
u_{i}^{r}(x)&=(|x|-r)u_{i}(x),\qquad
\nabla u_{i}^{r}(x)=(|x|-r)\nabla u_{i}(x)+\frac{x}{|x|}u_{i}(x),\qquad \text{ if \ }x\in B_{r+1}\smallsetminus B_{r}.
\end{align*}
Set $\delta:=\min\{\sigma_1,\ldots,\sigma_\ell,1\}$. Using that $|u_{i}\tfrac{x}{|x|}\cdot\nabla u_{i}|\leq\tfrac{1}{2}(|\nabla u_{i}|^{2}+|u_{i}|^{2})$ we obtain
\begin{align}
\int_{\mathbb{R}^{N}} \big(\nabla u_{i}\cdot\nabla u_{i}^{r}+V_i\,u_{i}u_{i}^{r}\big)&\geq \delta \xi_i(r+1) +\int_{B_{r+1}\smallsetminus B_{r}}\Big[(|x|-r)\left(|\nabla u_{i}|^{2}+V_i\,u_{i}^{2}\right)+u_{i}\frac{x}{|x|}\cdot\nabla u_{i}\Big]\nonumber\\
&\geq \delta \xi_i(r+1)-\frac{1}{2}\int_{B_{r+1}\smallsetminus B_{r}} \big(|\nabla u_{i}|^{2}+|u_{i}|^{2}\big) \nonumber\\
&\geq (\delta+\tfrac{1}{2})\xi_i(r+1)-\tfrac{1}{2}\xi_i(r)\qquad\text{if \ }r+1\geq\rho.\label{eq:bineq}
\end{align}
As $\bf u$ solves \eqref{eq:introsystem} we have that
\begin{align*}
\left\lvert\int_{\mathbb{R}^{N}} \nabla u_{i}\cdot\nabla u_{i}^{r}+V_i\,u_{i}u_{i}^{r}\right\rvert
&=\left\lvert\int_{\mathbb{R}^{N}}\sum_{j=1}^{\ell}\beta_{ij}|u_{j}|^{p}|u_{i}|^{p-2}u_{i}u_{i}^{r}\right\rvert\\
&\leq\sum_{j=1}^{\ell}\int_{\mathbb{R}^{N}\setminus B_{r}}|\beta_{ij}||u_{j}|^{p}|u_{i}|^{p-2}|u_{i}|^{2}
=\sum_{j=1}^{\ell}|\beta_{ij}|\int_{\mathbb{R}^{N}\smallsetminus B_{r}} |u_{j}|^{p}|u_{i}|^{p}
\end{align*}
and since $|u_{m}|^{p}\leq\left(\sum_{k=1}^{\ell}|u_{k}|^{2p}\right)^{1/2}$ for every $m=1,\ldots,\ell$, we obtain  
\begin{align*}
\left\lvert\int_{\mathbb{R}^{N}} \nabla u_{i}\cdot\nabla u_{i}^{r}+V_i\,u_{i}u_{i}^{r}\right\rvert\leq\left(\sum_{j=1}^{\ell}|\beta_{ij}|\right)\sum_{k=1}^{\ell}\int_{\mathbb{R}^{N}\smallsetminus B_{r}}|u_{k}|^{2p}.
\end{align*}
Given that $u_{k}\in H^{1}(\mathbb{R}^{N})$ for all $k=1,\ldots,\ell$,  Lemma \ref{independent} implies the existence of a constant $C_1=C_1(N,p)>0$ such that
\begin{align}
\left\lvert\int_{\mathbb{R}^{N}} \nabla u_{i}\cdot\nabla u_{i}^{r}+V_i\,u_{i}u_{i}^{r}\right\rvert\leq C_{1}\left(\sum_{j=1}^{\ell}|\beta_{ij}|\right)\sum_{k=1}^{\ell}\left(\int_{\mathbb{R}^{N}\smallsetminus B_{r}}\big(|\nabla u_{k}|^{2}+|u_{k}|^{2}\big)\right)^{p}\label{eq:upembed}
\end{align}
for every $r\geq 1$ and $i=1,\ldots,\ell$. Set $C_2:=C_1\sum_{i,j=1}^{\ell}|\beta_{ij}|.$ From \eqref{eq:bineq} and \eqref{eq:upembed}, assuming without loss of generality that $\rho\geq 2$ and adding over $i$, we get
\begin{align*}
\frac{2\delta+1}{2}\,|\bf\xi(r+1)|_1-\frac{1}{2}|\bf\xi(r)|_1\leq C_2\sum_{k=1}^{\ell}|\xi_{k}(r)|^{p}=:C_2\,|\bf\xi(r)|_{p}^{p}\qquad\text{if \ }r+1\geq\rho.
\end{align*}
Therefore,
\begin{align}
\frac{|\bf\xi(r+1)|_1}{|\bf\xi(r)|_1}\leq\frac{1}{2\delta+1}\left(1+2C_2\frac{|\bf\xi(r)|_{p}^{p}}{|\bf\xi(r)|_1}\right)\leq \frac{1}{2\delta+1}\left(1+2C_2|\bf\xi(r)|_{1}^{p-1}\right) =:\gamma(r)\quad\text{if \ }r+1\geq\rho.\label{eq:ackineq}
\end{align}
Since $|\bf\xi(r)|_{1}\to 0$ as $r\to\infty,$ there is $r_{0}=r_0(\bf u,p,\boldsymbol\beta,\rho)\in\n$ such that $r_0\geq\rho$ and $\gamma(r)\leq \gamma_0^{-1}$ for all $r\geq r_{0}$ with $\gamma_0:=\frac{2\delta+1}{\delta+1}>1$. Then, for $r>r_0+1$,
\begin{align*}
|\bf\xi(r)|_1\leq |\bf\xi(\lfloor r \rfloor)|_1=|\bf\xi(r_0)|_1\prod_{k=r_0}^{\lfloor r \rfloor-1}\frac{|\bf\xi(k+1)|_1}{|\bf\xi(k)|_1}
\leq |\bf\xi(r_0)|_1\gamma_0^{r_0-\lfloor r \rfloor}
\leq \|\bf u\|^2\gamma_0^{r_0-r+1},
\end{align*}
where $\lfloor r \rfloor$ denotes the floor of $r$. Since $|\bf\xi(r)|_1\leq \|\bf u\|^2\leq\|\bf u\|^2\gamma_0^{r_0-r+1}$ for $r\leq r_0+1$ we have that
$$|\bf\xi(r)|_1\leq\|\bf u\|^2\gamma_0^{r_0-r+1}=\|\bf u\|^2\gamma_0^{r_0+1} \e^{-\ln(\gamma_0)r}\qquad\text{for every \ }r\geq 0,$$
as claimed.
\end{proof}

\begin{lemma}
\label{sourceN}
Assume $(V_1)$ and let $\boldsymbol{u}=(u_1,\ldots,u_\ell)$ be a solution of \eqref{eq:introsystem}. Then $u_i\in W^{2,s}(\mathbb{R}^{N})\cap\mathcal{C}^{2}(\mathbb{R}^{N})$ for every $s\geq 2$ and $i=1,\ldots,\ell$.
\end{lemma}

\begin{proof}
Let $N\geq 3$. The argument for $N=1,2$ is similar and easier. For each $i=1,\ldots,\ell$ set
\begin{align} \label{eq:f_i}
f_{i}:=\sum_{j=1}^{l}\beta_{ij}|u_{j}|^{p}|u_{i}|^{p-2}u_{i}.
\end{align}
Since $|u_{k}|\leq |\boldsymbol{u}|:=\sqrt{u_1^2+\cdots+u_\ell^2}$ for every $k=1,\ldots\ell$, we have that
\begin{align}
\label{eq:fuestimate}
|f_{i}|\leq\sum_{i,j=1}^{\ell}|\beta_{ij}||u_{j}|^{p}|u_{i}|^{p-1}\leq\left(\sum_{j=1}^{\ell}|\beta_{ij}|\right)|\boldsymbol{u}|^{p}|\boldsymbol{u}|^{p-1}\leq\left(\sum_{i,j=1}^{\ell}|\beta_{ij}|\right)|\boldsymbol{u}|^{2p-1}.
\end{align}
Therefore, $f_{i}\in L^{s_{1}}(\mathbb{R}^{N})$ for $s_{1}:=\frac{2^{\ast}}{2p-1}>1$ and, by the standard $L^{p}$-elliptic regularity theory, $u_{i}\in W^{2,s_{1}}(\mathbb{R}^{N})$ for all $i=1,\ldots,\ell$ (see, e.g., \cite[Chapter 9]{gilbarg} or \cite[Section 3.2]{villavert}). Using a bootstrapping argument, we conclude the existence of $s>\max\{\frac{N}{2},2\}$ such that $u_i\in W^{2,s}(\mathbb{R}^{N})$ for all $i=1,\ldots,\ell$ and thus, by the Sobolev embedding theorem, $u_i\in \cC^{1,\alpha}(\rn)$. Since $V_i$ is Hölder continuous and bounded, applying the Schauder estimates repeatedly, we deduce that $u_i$ is of class $\cC^{2}$ (see \cite[Section 1.3]{han2016}).
\end{proof}

In the rest of the paper, we write $|\cdot|_t$ for the norm in $L^t(\rn)$, $1\leq t\leq\infty$.  If $\boldsymbol{u}=(u_1,\ldots,u_\ell)\in [L^\infty(\R^N)]^\ell$, then $|\boldsymbol{u}|_\infty:=\sum_{i=1}^\ell \sup_{\R^N}|u_i|$. Moreover, for a proper open subset $\o$ of $\rn$ we denote the usual Sobolev norm in $H^1(\o)$ by $\|\cdot\|_{H^1(\o)}$, i.e.,
$$\|u\|_{H^1(\o)}^2:=\io (|\nabla u|^2+|u|^2).$$

\begin{lemma} \label{lem:estimate}
Assume $(V_1)$. Let $\boldsymbol{u}=(u_1,\ldots,u_\ell)$ be a solution of \eqref{eq:introsystem}, $s>\max\{2,\frac{N}{2}\}$ and $\Lambda>0$ be such that $|V_i|_\infty\leq \Lambda$ for $i=1,\ldots,\ell$. Then there is a constant $C=C(\boldsymbol \beta, N, p, \Lambda,s)>0$ such that, for any $x\in\rn$,
\begin{align*}
\|u_{i}\|_{W^{2,s}(B_\frac{1}{2}(x))}\leq C\left(|u_{i}|_\infty^{\frac{s-2}{s}}\|u_{i}\|_{H^1(B_1(x))}^{\frac{2}{s}}+|\boldsymbol{u}|^{\frac{2ps-(s+2)}{s}}_\infty\Big(\sum_{j=1}^\ell \|u_j\|_{H^1(B_1(x))}^2\Big)^{\frac{p}{s}}\right),
\end{align*}
where $|\boldsymbol{u}|:=\sqrt{u_1^2+\cdots+u_\ell^2}$ and $B_R(x)$ is the ball of radius $R$ centered at $x$.
\end{lemma}

\begin{proof}
Since $u_i\in W^{2,s}(\rn)\subset L^{\infty}(\mathbb{R}^{N})$, we have that
\begin{align*}
|u_{i}|^{s}=|u_{i}|^{s-2}|u_{i}|^{2}\leq|u_{i}|_\infty^{s-2}|u_{i}|^{2}.
\end{align*}
Set $f_i$ as in \eqref{eq:f_i}. By \eqref{eq:fuestimate}, there is a constant $C_2=C_2(\boldsymbol\beta)$ such that
\begin{align*}
|f_{i}|^{s}&\leq C_2^{s}|\boldsymbol{u}|^{(p-1)s}|\boldsymbol{u}|^{ps}=C_2^{s}|\boldsymbol{u}|^{(p-1)s+p(s-2)}( u_1^2+\cdots+u_\ell^2)^{p}\\
&\leq C_2^{s}|\boldsymbol{u}|_\infty^{2ps-(s+2)}\ell^p(u_1^{2p}+\cdots+u_\ell^{2p}),
\end{align*}
where $(p-1)s+p(s-2)>0$. Then, by \cite[Theorem 9.11]{gilbarg}, there is a positive constant $C_1=C_1(s,N,\Lambda)$ such that
\begin{align*}
\|u_{i}\|_{W^{2,s}(B_\frac{1}{2}(x))}\leq C_1\left(|u_{i}|_{L^s(B_1(x))}+|f_{i}|_{L^s(B_1(x))}\right)\quad\text{for any \ }x\in\rn.
\end{align*}From the previous inequalities we derive
\begin{align*}
\|u_{i}\|_{W^{2,s}(B_\frac{1}{2}(x))}\leq C_1\left(|u_{i}|^{\frac{s-2}{s}}_\infty\|u_{i}\|_{H^1(B_1(x)))}^{\frac{2}{s}}+C_2\ell^\frac{p}{s}C_3|\boldsymbol u|_\infty^{\frac{2ps-(s+2)}{s}}\Big(\sum_{j=1}^\ell \|u_j\|_{H^1(B_1(x))}^2\Big)^{\frac{p}{s}}\right),
\end{align*}
where $C_3=C_3(N,p)$ is the constant given by the Sobolev embedding $H^1(B_1)\subset L^{2p}(B_1)$.
\end{proof}

\begin{lemma} \label{expdecay1}
Assume $(V_1)-(V_2)$, let $\boldsymbol{u}=(u_1,\ldots,u_\ell)$ be a solution of \eqref{eq:introsystem} and let $f_i$ be as in \eqref{eq:f_i}. Then, there are constants $\eta>0$, $C_1>0$, and $C_2>0$ such that
\begin{align*}
|u_i(x)|\leq C_1 e^{-\eta|x|},\qquad 
|f_{i}(x)|\leq C_2 e^{-(2p-1)\eta |x|},\qquad\mbox{for all}~x\in\mathbb{R}^{N}\text{ and }i=1,\ldots,\ell.
\end{align*}
\end{lemma}

\begin{proof}
For $x\in\rn$ with $|x|\geq 2$, set $r:=\frac{1}{2}|x|$. Then, $B_1(x)\subset\rn\smallsetminus B_r$ and, by Lemma \ref{normdecay}, there are positive constants $K_1=K_1(\boldsymbol u,\boldsymbol\sigma,\boldsymbol\beta,N,\rho,p)$ and $\vartheta=\vartheta(\boldsymbol\sigma)$, with $\rho$ and $\sigma_i$ as in $(V_2)$, such that
$$\|u_j\|_{H^1(B_1(x))}^2\leq\|u_j\|_{H^1(\rn\smallsetminus B_r)}^2=\xi_j(r)\leq \sum_{i=1}^\ell\xi_i(r)\leq K_1\e^{-\vartheta r}\quad\text{for every \ }j=1,\ldots,\ell.$$
Fix $s>\max\{\frac{N}{2},2\}$. By Lemma \ref{lem:estimate} there are positive constants $K_2=K_2(\bf u, \boldsymbol\beta, N, p, \Lambda,s)$ and $K_3=K_3(\boldsymbol u,\boldsymbol\sigma,\boldsymbol\beta,\rho,N,p,s)$ such that
$$\|u_{i}\|_{W^{2,s}(B_\frac{1}{2}(x))}\leq K_2\left(\|u_{i}\|_{H^1(B_1(x)))}^{\frac{2}{s}}+\Big(\sum_{j=1}^\ell \|u_j\|_{H^1(B_1(x))}^2\Big)^{\frac{p}{s}}\right)\leq K_2K_3\e^{-\frac{\vartheta}{s}r}.$$
Therefore,
$$|u_i(x)|\leq |u_i|_{L^\infty(B_\frac{1}{2}(x))}\leq K_4\|u_{i}\|_{W^{2,s}(B_\frac{1}{2}(x))}\leq K_2K_3K_4\e^{-\frac{\vartheta}{2s}|x|}\quad\text{for every \ }x\in\rn\smallsetminus B_2,$$
where $K_4$ is the positive constant given by the embedding $W^{2,s}(B_\frac{1}{2})\subset L^\infty(B_\frac{1}{2})$. Since $u_i$ is continuous, we may choose $C_1\geq K_2K_3K_4$ such that $|u_i(x)|\leq C_1\e^{-\frac{\vartheta}{s}}$ for every $x\in B_2$. So, setting $\eta:=\frac{\vartheta}{2s}$, we obtain
$$|u_i(x)|\leq C_1\e^{-\eta|x|}\quad\text{for every \ }x\in\rn.$$
The estimate for $f_i$ follows immediately from \eqref{eq:fuestimate}.
\end{proof}

The following result is a particular case of \cite[Theorem 2.1]{RS00}. We include a simplified proof for completeness.

\begin{lemma}\label{lem:pm}
Assume that $V:\R^N\to\R$ satisfies $\sigma:=\inf_{\R^N\backslash B_{\rho}(0)} V>0$ for some $\rho\geq 0$. Let $w$ be a classical solution of $-\Delta w + V w = f$ in $\R^N$ such that
\begin{align*}
|w(x)|\leq C \e^{-\eta|x|}\quad \text{ and }\quad |f(x)|\leq C \e^{-\delta|x|}\qquad \text{ for all }x\in\R^N
\end{align*}
and for some constants $C>0$, $\eta\in(0,\sqrt{\sigma})$ and $\delta\in(\eta,\sqrt{\sigma}]$. Then, for any $\mu\in(\eta,\delta)$, there is $M=M(\mu,\delta,\rho,\sigma,C)>0$ such that 
\begin{align*}
|w(x)|\leq M \e^{-\mu|x|}\qquad \text{for all \ }x\in\R^N.
\end{align*}
\end{lemma}

\begin{proof}
Let $\rho,\sigma,\eta,\delta,\mu,$ and $C$ be as in the statement. Set $v(x):=\e^{-\mu|x|}$ for $x\in\R^N$.  Then, 
$$\Delta v(x)=v(x) h(|x|)\quad\text{for \ }x\in \R^N\smallsetminus\{0\},\qquad\text{where \ } h(r):=\mu^2-(N-1)\frac{\mu}{r}.$$
In particular,  $V(x)-h(|x|)\geq \sigma-\mu^2=:\eps>0$ for $|x|>\rho$. Fix $t\in\r$ satisfying
\begin{align}\label{teq}
t>\frac{C}{\eps} \e^{(\mu-\delta)\rho}\qquad \text{and}\qquad w(x)<tv(x)\text{ \ for \ }|x|=\rho.
\end{align}
 We claim that $w(x)\leq tv(x)$ for all $|x|>\rho$. Indeed, let $z:=w-tv$ and assume, by contradiction, that $m:=\sup_{|x|\geq \rho}z(x)>0$. Since $\lim_{|x|\to\infty}z(x)=0$, there is $R>\rho$ such that $z(x)\leq \frac{m}{2}$ for $|x|\geq R$. Let $\Omega:=\{x\in\R^N\::\: \rho<|x|<R\text{ and }z(x)>0\}$. Then $z\leq \frac{m}{2}$ on $\partial \Omega$ and, by \eqref{teq},
\begin{align*}
-\Delta z(x)&=-\Delta w(x)+t\Delta v(x)
= f(x)- V(x)w(x)+tv(x)h(|x|)\\
&= f(x)- V(x)z(x)+tv(x)(h(|x|)-V(x))\\
&<C \e^{-\delta|x|}-\eps tv(x)
=C \e^{-\delta|x|}-\eps t \e^{-\mu|x|}<0\quad\text{ \ for every \ }x\in\o.
\end{align*}
Then, by the maximum principle, $m=\max_{\Omega}z = \max_{\partial \Omega}z\leq\frac{m}{2}$. This is a contradiction. Therefore $m\leq 0$, namely, $w(x)\leq t \e^{-\mu|x|}$ for all $|x|\geq \rho$. Arguing similarly for $-w$ and using that $w\in L^\infty(\R^N)$ we obtain that $|w(x)|\leq M\e^{-\mu|x|}$ for all $x\in\rn$, as claimed.
\end{proof}

We are ready to prove Theorem \ref{mainresult}.

\begin{proof}[Proof of Theorem \ref{mainresult}]
Iterating Lemmas \ref{expdecay1} and \ref{lem:pm}, using that $2p-1>1$, one shows that, for any $\mu_i\in(0,\sqrt{\sigma_i})$, there is $C>0$ such that  $|u_{i}(x)|\leq C e^{-\mu_i|x|}$ for all $x\in\R^N$ and for all $i=1,\ldots,\ell$.  

Now, assume that $V_i\equiv 1$ for every $i=1,\ldots,\ell$ and let $\mu\in(0,1)$ be such that $(2p-1)\mu>1$.  By Lemma \ref{expdecay1}, we have that $|f_{i}(x)|\leq C_2 \e^{-(2p-1)\mu |x|}$ for all $x\in\mathbb{R}^{N}$. The claim now follows from \cite[Theorem 2.3$(c)$]{ackermann2016}.
\end{proof}

\section{Energy estimates for seminodal solutions}\label{ee:sec}

In this section we prove Theorem \ref{thm:main2}. Consider the autonomous system \eqref{eq:system} where $N\geq 4$, $1<p<\frac{N}{N-2}$ and $\beta_{ij}$ satisfy the assumption $(B_1)$ stated in the Introduction. According to the decomposition given by $(B_1)$, a solution $\bf u=(u_1,\ldots,u_\ell)$ to \eqref{eq:system} may be written in block-form as
\[\bf u=(\bar u_1,\ldots,\bar u_q)\qquad\text{with \ }\bar u_h=(u_{\ell_{h-1}+1},\ldots,u_{\ell_h}),\quad h=1,\ldots,q.\]
$\bf u$ is called \emph{fully nontrivial} if every component $u_i$ is different from zero. We say that $\bf u$ is \emph{block-wise nontrivial} if at least one component in each block $\bar u_h$ is nontrivial.

Following \cite{clapp2022}, we introduce suitable symmetries to produce a change of sign in some components. Let $G$ be a finite subgroup of the group $O(N)$ of linear isometries of $\rn$ and denote by $Gx:=\{gx:g\in G\}$ the $G$-orbit of $x\in\rn$. Let $\phi:G\to\z_2:=\{-1,1\}$ be a homomorphism of groups. A function $u:\rn\to\r$ is called \emph{$G$-invariant} if it is constant on $Gx$ for every $x\in\rn$ and it is called \emph{$\phi$-equivariant} if
\begin{equation} \label{eq:equivariant}
u(gx)=\phi(g)u(x) \text{ \ for all \ }g\in G, \ x\in\rn.
\end{equation}
Note that, if $\phi\equiv 1$ is the trivial homomorphism and $u$ satisfies \eqref{eq:equivariant}, then $u$ is $G$-invariant. On the other hand, if $\phi$ is surjective every nontrivial function satisfying \eqref{eq:equivariant} is nonradial and changes sign. Define 
$$H^1(\rn)^\phi:=\{u\in H^1(\rn): u\text{ is }\phi\text{-equivariant}\}.$$
For each $h=1,\ldots,q$, fix a homomorphism $\phi_h:G\to\z_2$. Take $\phi_i := \phi_h$ for all $i\in I_h$ and set $\bf\phi=(\phi_1,\ldots,\phi_\ell)$. Denote by
\begin{align*}
 \cH^{\bf\phi} &:=H^1(\rn)^{\phi_1}\times\cdots\times H^1(\rn)^{\phi_\ell},
\end{align*}
and let $\cJ^{\bf\phi}:\cH^{\bf\phi}\to\r$ be the functional given by 
$$\cJ^{\bf\phi}(\bf u) := \frac{1}{2}\dsum_{i=1}^{\ell}\|u_i\|^2 - \frac{1}{2p}\dsum_{i,j=1}^\ell\beta_{ij}\irn |u_i|^p|u_j|^p.$$
This functional is of class $\cC^1$ and its critical points are the solutions to the system \eqref{eq:system} satisfying \eqref{eq:equivariant}. The block-wise nontrivial solutions belong to the Nehari set
$$\cN^{\bf\phi}:= \{\bf u\in\cH^{\bf\phi}:\|\bar u_h\|\neq 0\text{ \ and \ }\partial_{\bar u_h}\cJ^{\bf\phi}(\bf u)\bar u_h=0 \text{ \ for every \ } h=1,\ldots,\ell\}.$$
Note that
\begin{align*}
&\partial_{\bar u_h}\cJ^{\bf\phi|K}(\bf u)\bar u_h=\|\bar u_h\|^2 - \sum_{k=1}^\ell\,\sum_{(i,j)\in I_h\times I_k}\beta_{ij}\irn|u_i|^p|u_j|^p,
\end{align*}
and that \ $\cJ^\phi(\bf u)=\frac{p-1}{2p}\|\bf u\|^2$ if $u\in\cN^\phi$. \ Let
\begin{equation*}
c^{\bf\phi}:= \inf_{\bf u\in \cN^{\bf\phi}}\cJ^{\bf\phi}(\bf u).
\end{equation*} 
If $\bf s=(s_1,\ldots,s_q)\in\r^q$ and $\bf u=(\bar{u}_{1},\ldots,\bar{u}_{q})\in \cH^{\bf\phi}$ we write $\bf s\bf u:=(s_1\bar{u}_{1},\ldots,s_q\bar{u}_{q})$. The following facts were proved in \cite{cp}.

\begin{lemma} \label{lem:nehari}
\begin{itemize}
\item[$(i)$] $c^{\bf\phi} >0$.
\item[$(ii)$] If the coordinates of $\bf u\in \cH^{\bf\phi}$ satisfy
\begin{equation} \label{eq:N}
\sum_{k=1}^q\,\sum_{(i,j)\in I_h\times I_k}\irn\beta_{ij}|u_i|^p|u_j|^p>0\qquad\text{for every \ }h=1,\ldots,q,
\end{equation}
then there exists a unique $\bf s_{\bf u} \in (0,\infty)^{q}$ such that $\bf s_{\bf u}\bf u \in \cN^{\bf\phi}$. Furthermore,  
$$\cJ^{\bf\phi}(\bf s_{\bf u} \bf u) = \max_{{\bf s} \in (0,\infty)^q}\cJ^{\bf\phi}(\bf s \bf u).$$
\end{itemize}
\end{lemma}

\begin{proof}
See \cite[Lemma 2.2]{cp} or \cite[Lemma 2.2]{clapp2022}.
\end{proof}

\begin{lemma} \label{lem:existence}
If $c^{\bf\phi}$ is attained, then the system \eqref{eq:system} has a block-wise nontrivial solution $\bf u=(u_1,\ldots,u_\ell)\in \cH^{\bf\phi}$. Furthermore, if $u_i$ is nontrivial, then $u_i$ is positive if $\phi_i\equiv 1$ and $u_i$ is nonradial and changes sign if $\phi_i$ is surjective.
\end{lemma}

\begin{proof}
It is shown in \cite[Lemma 2.4]{cp} that any minimizer of $\cJ^{\bf\phi}$ on $\cN^{\bf\phi}$ is a block-wise nontrivial solution to \eqref{eq:system}. If $u_i\neq 0$ and $\phi_i$ is surjective, then $u_i$ is nonradial and changes sign. If $\phi_i\equiv 1$ then $|u_i|$ is $G$-invariant and replacing $u_i$ with $|u_i|$ we obtain a solution with the required properties.
\end{proof}

Set $Q:=\{1,\ldots,q\}$ and fix a decomposition $Q=Q^+\cup Q^-$ with $Q^+\cap Q^-=\emptyset$. From now on, we consider the following symmetries. We write $\rn\equiv\cc\times\cc\times\r^{N-4}$ and a point in $\rn$ as $(z_1,z_2,y)\in\cc\times\cc\times\r^{N-4}$. 

\begin{definitions} \label{definition}
Let $\mathrm{i}$ denote the imaginary unit.  For each $m\in\n$, let
\begin{align*}
K_m:=\{\mathrm{e}^{2\pi\mathrm{i}j/m}:j=0,\ldots,m-1\},
\end{align*}
$G_m$ be the group generated by $K_m\cup\{\tau\}\cup O(N-4)$, acting on each point $(z_1,z_2,y)\in\cc\times\cc\times\r^{N-4}$ as
\begin{align*}
\mathrm{e}^{2\pi\mathrm{i}j/m}(z_1,z_2,y)&:=(\mathrm{e}^{2\pi\mathrm{i}j/m}z_1,\mathrm{e}^{2\pi\mathrm{i}j/m}z_2,y), \qquad\qquad\tau(z_{1},z_{2},y):=(z_{2},z_{1},y),\\
\alpha(z_1,z_2,y)&:=(z_1,z_2,\alpha y)\quad\text{if \ }\alpha\in O(N-4),
\end{align*}
and $\theta:G_m\to\z_2$ be the homomorphism satisfying 
\begin{align*}
\theta(\mathrm{e}^{2\pi\mathrm{i}j/m})=1,\quad \theta(\tau)=-1,\quad \text{ and }\quad \theta(\alpha)=1\quad \text{ for every $\alpha\in O(N-4)$.}    
\end{align*}
Define $\phi_h:G_m\to\z_2$ by
\begin{equation} \label{eq:phi}
\phi_h:=
\begin{cases}
1 &\text{if \ }h\in Q^+,\\
\theta &\text{if \ }h\in Q^-.
\end{cases}
\end{equation}
\end{definitions}

Due to the lack of compactness, $c^{\bf\phi}$ is not always attained; see e.g. \cite[Corollary 2.8$(i)$]{clapp2022}. A sufficient condition for this to happen is given by the next lemma. We use the following notation. If $Q'\subset Q:=\{1,\ldots,q\}$ we consider the subsystem of \eqref{eq:system} obtained by deleting all components of $\bar{u}_h$ for every $h\notin Q'$, and we denote by $\cJ^{\bf\phi}_{Q'}$ and $\cN^{\bf\phi}_{Q'}$ the functional and the Nehari set associated to this subsystem. We write
\begin{align*}
 c_{Q'}^{\bf\phi}:= \inf_{\bf u\in \cN_{Q'}^{\bf\phi}}\cJ_{Q'}^{\bf\phi}(\bf u).
\end{align*}
If $Q'=\{h\}$ we omit the curly brackets and write, for instance, $c_{h}^{\bf\phi}$ or $\cJ^{\bf\phi}_{h}.$

\begin{lemma}[\textbf{Compactness}] \label{lem:compactness}
Let $N\neq 5$, $m\geq 5$ and $\phi_h:G_m\to\z_2$ be as in \eqref{eq:phi}. If, for each $h\in Q:=\{1,\ldots,q\}$, the strict inequality
\begin{equation} \label{eq:compactness}
c^{\bf\phi}<
\begin{cases}
c^{\bf\phi}_{Q\smallsetminus\{h\}}+m\mu_h\frac{p-1}{2p}\|\omega\|^2, &\text{if \ }h\in Q^+,\\ \smallskip
c^{\bf\phi}_{Q\smallsetminus\{h\}}+2m\mu_h\frac{p-1}{2p}\|\omega\|^2, &\text{if \ }h\in Q^-,
\end{cases}
\end{equation}
holds true, then $c^{\bf\phi}$ is attained, where $\omega$ is the positive radial solution to \eqref{eq:omega} and $\mu_h$ is given by \eqref{eq:mu_h}.
\end{lemma}

\begin{proof}
This statement follows by combining \cite[Corollary 2.8$(ii)$]{clapp2022} with \cite[Equation (5.1)]{clapp2022}.
\end{proof}

To verify condition \eqref{eq:compactness} we introduce a suitable test function. Fix $m\geq 5$ and let $K_m$ be as in Definitions \ref{definition}. If $h\in Q^+$, we take $\zeta_h:=(\frac{1}{\sqrt{2}},\frac{1}{\sqrt{2}},0)$ and, for each $R>1$, we define
$$\widehat\sigma_{hR}(x):=\sum_{g\in K_m}\omega(x-Rg\zeta_h),\qquad x\in\rn.$$
If $h\in Q^-$ we take $\zeta_h:=(1,0,0)$ and we define
$$\widehat\sigma_{hR}(x):=\sum_{g\in G_m'}\phi_h(g)\,\omega(x-Rg\zeta_h),\qquad x\in\rn,$$
where $\omega$ is the positive radial solution to \eqref{eq:omega} and $G'_m$ is the subgroup of $G_m$ generated by $K_m\cup\{\tau\}$. Note that $\widehat\sigma_{hR}(gx)=\phi_h(g)\widehat\sigma_{hR}(x)$ for every $g\in G_m$, $x\in\rn$. Let
\begin{equation} \label{eq:sigma}
\sigma_{hR}:=t_{hR}\widehat\sigma_{hR},
\end{equation}
where  $t_{hR}>0$ is chosen so that $\|\sigma_{hR}\|^2=\irn|\sigma_{hR}|^{2p}$.

\begin{lemma} \label{lem:test}
If $m\geq 5$, then, for each $h\in\{1,\ldots,q\}$, there exist $\bar t_{h}=(t_{\ell_{h-1}+1},\ldots,t_{\ell_h})\in(0,\infty)^{\ell_h-\ell_{h-1}}$ and $C_0,R_0>0$ such that $\bar t_{h}\sigma_{hR}:=(t_{\ell_{h-1}+1}\sigma_{hR},\ldots,t_{\ell_h}\sigma_{hR})\in\cN^{\bf\phi}_h$ and 
$$\cJ^{\bf\phi}_{h}(\bar t_{h}\sigma_{hR})\leq |G_m\zeta_h|\,\mu_h\tfrac{p-1}{2p}\|\omega\|^2-C_0\mathrm{e}^{-Rd_m}\qquad\text{for every \ }R\geq R_0,$$
where $|G_m\zeta_h|$ is the cardinality of the $G_m$-orbit of $\zeta_h$, i.e., $|G_m\zeta_h|=m$ if $h\in Q^+$ and $|G_m\zeta_h|=2m$ if $h\in Q^-$, and
\begin{align}\label{dm}
d_m:=|1-\mathrm{e}^{2\pi\mathrm{i}/m}|.
\end{align}
\end{lemma}

\begin{proof}
Take $\bar t_h=(t_{\ell_{h-1}+1},\ldots,t_{\ell_h})\in(0,\infty)^{\ell_h-\ell_{h-1}}$ such that 
$$\sum_{i\in I_h}t_i^2=\sum_{i,j\in I_h}\beta_{ij}t_j^pt_i^p=\mu_h$$
and apply \cite[Proposition 4.1$(i)$ and Lemma 4.4]{clapp2022}.
\end{proof}
\smallskip

\begin{proof}[Proof of Theorem \ref{thm:main2}]
Assume $(B_1)$ and let $\phi_h:G_m\to\z_2$ be given by \eqref{eq:phi}. For $q=1$ and $m\geq 5$ it is proved in \cite[Corollary 4.2 and Proposition 4.5]{clapp2022} that $c^{\bf\phi}$ is attained at $\bf u\in\cN^{\bf\phi}$ satisfying
$$\|\bf u\|^2=\mu_1\|\omega\|^2\text{ \ if \ }Q^+=\{1\}\qquad\text{and}\qquad\|\bf u\|^2<2m\,\mu_1\|\omega\|^2\text{ \ if \ }Q^-=\{1\}.$$
Taking $m=5$ gives statement $(b)$.

Fix $m=6$. We claim that $c^{\bf\phi}$ is attained and that the estimate $(c)$ holds true for every $q\geq 2$. To prove this claim, we proceed by induction. Assume it is true for $q-1$ with $q\geq 2$. 

We will show that the compactness condition \eqref{eq:compactness} holds true.  Using a change of coordinates, it suffices to argue for $h=q$. By induction hypothesis there exists $\bf w=(\bar w_1,\ldots,\bar w_{q-1})\in\cN^{\bf\phi}_{Q\smallsetminus\{q\}}$ such that $\cJ^{\bf\phi}_{Q\smallsetminus\{q\}}(\bf w)=c^{\bf\phi}_{Q\smallsetminus\{q\}}$. For each $R>1$ let $\sigma_{qR}$ be as in \eqref{eq:sigma} and take $\bar t_q\in(0,\infty)^{\ell-\ell_{q-1}}$ as in Lemma \ref{lem:test}. Set $\bar w_{hR}=\bar w_h$ for $h=1,\ldots,q-1$ and $\bar w_{qR}=\bar t_q\sigma_{qR}$, and define $\bf w_R=(w_{1R},\ldots,w_{\ell R}):=(\bar w_{1R},\ldots,\bar w_{qR})$.  Then, as $\bf w\in\cN^{\bf\phi}_{Q\smallsetminus\{q\}}$ and the interaction between the components of $\bf w$ and $\sigma_{qR}$ tends to $0$ as $R\to\infty$, we have that $\bf w_R$ satisfies \eqref{eq:N} for large enough $R$ and, as a consequence, there exist $R_1>0$ and $(s_{1R},\ldots,s_{qR})\in[{1}/{2},2]^q$ such that $(s_{1R}\bar w_{1R},\ldots,s_{qR}\bar w_{qR})\in\cN^{\bf\phi}$ if $R\geq R_1$. Set $\bf u_R=(u_{1R},\ldots,u_{\ell R}):=(s_{1R}\bar w_{1R},\ldots,s_{qR}\bar w_{qR})$. Using that $\bf w\in\cN^{\bf\phi}_{Q\smallsetminus\{q\}}$ and $\bar t_q\sigma_{qR}\in\cN_q^{\bf\phi}$, from the last statement in Lemma \ref{lem:nehari}$(ii)$ and Lemma \ref{lem:test} we derive
\begin{align*}
\cJ^{\bf\phi}(\bf u_R) &= \frac{1}{2}\sum_{i=1}^\ell\|u_{iR}\|^2 - \frac{1}{2p}\sum_{i,j=1}^\ell\beta_{ij}\irn |u_{iR}|^p|u_{jR}|^p \\
&\leq\cJ^{\bf\phi}_{Q\smallsetminus\{q\}}(\bf w)+\cJ_q^{\bf\phi}(\bar t_q\sigma_{qR}) - \frac{1}{p}\sum_{h=1}^{q-1}\,\sum_{(i,j)\in I_h\times I_q}\beta_{ij}\irn |s_{hR}w_{iR}|^p|s_{qR}w_{jR}|^p \\
&\leq c^{\bf\phi}_{Q\smallsetminus\{q\}} + |G_m\zeta_h|\,\mu_q\tfrac{p-1}{2p}\|\omega\|^2 - C_0\mathrm{e}^{-Rd_m} + C_1 \sum_{h=1}^{q-1}\,\sum_{i\in I_h}\irn |w_{iR}|^p|\sigma_{qR}|^p,
\end{align*}
if $R\geq\max\{R_0,R_1\}$, where $C_0$ and $C_1$ are positive constants and $d_m$ is given in \eqref{dm}.

It is well known that  $|\omega(x)|\leq C\mathrm{e}^{-|x|}$ and, as $\bf w$ solves a subsystem of \eqref{eq:system}, Theorem \ref{mainresult} asserts that
\begin{equation*}
|w_{iR}(x)|\leq C\mathrm{e}^{-|x|}\quad\text{ \ for every \ }i\in I_h \text{ \ with \ }h=1,\ldots, q-1.
\end{equation*}
Therefore, for every $g\in G_m$,
$$\irn |w_{iR}|^p|\omega(\,\cdot\,-Rg\zeta_h)|^p \leq C\irn \mathrm{e}^{-p|x|}\,\mathrm{e}^{-p|x-Rg\zeta_h|}\d x \leq C\mathrm{e}^{-Rp}.$$
So, if $p>d_m$, we conclude that
$$c^{\bf\phi}< c^{\bf\phi}_{Q\smallsetminus\{q\}} + |G_m\zeta_h|\,\mu_q\tfrac{p-1}{2p}\|\omega\|^2$$
and, by Lemmas \ref{lem:compactness} and \ref{lem:existence}, $c^{\bf\phi}$ is attained at a block-wise nontrivial solution $\bf u$ of \eqref{eq:system} such that every component of $\bar u_h$ is positive if $h\in Q^+$ and every component of $\bar u_h$ is nonradial and changes sign if $h\in Q^-$. Furthermore, since we are assuming $(B_2)$ and $(B_3)$ with $C_*$ as in \eqref{C} below, \cite[Theorem 3.3]{clapp2022} asserts that $\bf u$ is fully nontrivial.

Finally, note that $p>1=d_m$ because $m=6$. As $|G_m\zeta_h|=6$ if $h\in Q^+$ and $|G_m\zeta_h|=12$ if $h\in Q^-$, the estimate in statement $(c)$ follows by induction.
\end{proof}

\begin{remark}
\emph{If $m=5$ and $p>d_m$ we arrive to a similar conclusion, where, in this case, the constant $b_h$ in statement $(b)$ is $5$ if $h\in Q^+$ and it is $10$ if $h\in Q^-$. Note, however, that numbers $p$ satisfying $d_5=2\sin\frac{\pi}{5}<p<\frac{N}{N-2}$ exist only for $N\leq 13$.}
\end{remark}

\begin{remark}\label{C:rmk}
\emph{
For $\phi_h$ as in \eqref{eq:phi}, the constant $C_*>0$ appearing in $(B_3)$ depends on $N$, $p$, $q$, and $Q^+$. It is explicitly defined in \cite[Equation (3.1)]{clapp2022} as
\begin{align}\label{C}
C_*:=\left(\frac{pd_{\bf\phi}}{(p-1)S_{\bf\phi}^{\frac{p}{p-1}}}\right)^p,
\end{align}
where
\begin{align*}
d_\phi:=\frac{p-1}{2p}\inf_{(v_1,\ldots,v_q)\in \cU^{\bf \phi}}\sum_{h=1}^q\|v_h\|^2
\end{align*}
with \ $\cU^{\bf\phi}:=\{(v_1,\ldots,v_q)\::\: v_h\in H^1(\R^N)^{\phi_h}\backslash \{0\},\ \|v_h\|^2=|v_h|_{2p}^{2p},\ v_hv_k=0\ \text{ if }h\neq k\}$, \ and 
\begin{align*}
S_{\bf\phi}:=\min_{h=1,\ldots,q}\,\inf_{v\in H^1(\R^N)^{\phi_h}\backslash \{0\}}\frac{\|v\|^2}{|v|_{2p}^2}.
\end{align*}
}
\end{remark}

\begin{remark}
\emph{In the proof of Theorem \ref{thm:main2} we use \cite[Theorem 2.3]{ackermann2016}, which also characterizes the \emph{sharp} decay rate for \emph{positive} components by providing a bound from below.  This kind of information can be useful to show uniqueness of positive solutions for some problems, see \cite[Section 8.2]{BFSST18}.}
\end{remark}

To conclude, we discuss some special cases.

\begin{examples}
Assume $(B_1)$ and let $p\in(1,\frac{2^*}{2})$.
\begin{itemize}
\item[$(a)$] If $q=1$ the system \eqref{eq:system} is cooperative and more can be said. Indeed, it is shown in \emph{\cite[Corollary 4.2 and Proposition 4.5]{clapp2022}} that, if $(B_2)$ is satisfied, then \eqref{eq:system} has a synchronized solution $\bf u=(t_1u,\ldots,t_\ell u)$, where $(t_1,\ldots,t_\ell)\in(0,\infty)^\ell$ is a minimizer for \eqref{eq:mu_h} and $u$ is a nontrivial $\phi$-equivariant least energy solution of the equation
\begin{align}\label{O}
-\Delta u+u=|u|^{2p-2}u,\qquad u\in H^1(\rn)^\phi.    
\end{align}
Here, if $Q^+=\{1\}$, then $\phi\equiv 1$ (and therefore $u=\omega$) and $\|\bf u\|^2\leq \mu_1\|\omega\|^2$.  On the other hand, if  $Q^-=\{1\}$, then $\phi:G_m\to\z_2$ is the homomorphism $\theta$ given in \emph{Definitions~\ref{definition}} and $\|\bf u\|^2\leq 10\mu_1\|\omega\|^2$.

\item[$(b)$] If $q=\ell\geq 2$ the system \eqref{eq:system} is competitive, i.e., $\beta_{ii}>0$ and $\beta_{ij}<0$ if $i\neq j$. Assumptions $(B_2)$ and $(B_3)$ are automatically satisfied and, as $\mu_i=\beta_{ii}^{-\frac{1}{p-1}}$, the estimate in \emph{Theorem \ref{thm:main2}$(c)$} becomes
\begin{align*}
\|\bf u\|^2&< \left(\min_{j\in Q}\Big(a_j\beta_{jj}^{-\frac{1}{p-1}}+\sum_{i\in Q\smallsetminus\{i\}}b_i\beta_{ii}^{-\frac{1}{p-1}}\Big)\right)\|\omega\|^2\\
&\leq
\begin{cases}
\left(6\,|Q^+|+12\,|Q^-|-5\right)\beta_0^{-\frac{1}{p-1}}\|\omega\|^2 &\text{if \ }Q^+\neq\emptyset,\\
12\,|Q^-|\beta_0^{-\frac{1}{p-1}}\|\omega\|^2 &\text{if \ }Q^+=\emptyset,
\end{cases}
\end{align*}
where $|Q^\pm|$ denotes the cardinality of $Q^\pm$ and $\beta_0:=\min\{\beta_{11},\ldots,\beta_{\ell\ell}\}$.

\item[$(c)$] Similarly, for any $q\geq 2$, the estimate in \emph{Theorem \ref{thm:main2}$(c)$} yields
\begin{equation*}
\|\bf u\|^2\leq
\begin{cases}
\left(6\,|Q^+|+12\,|Q^-|-5\right)\,\mu_*\|\omega\|^2 &\text{if \ }Q^+\neq\emptyset,\\
12\,|Q^-|\,\mu_*\|\omega\|^2 &\text{if \ }Q^+=\emptyset.
\end{cases}
\end{equation*}
where $\mu_*=\max\{\mu_1,\ldots,\mu_q\}$. 

Assumptions $(B_2)$ and $(B_3)$ guarantee that $\bf u$ is fully nontrivial. Note that the left-hand side of the inequality in $(B_3)$ depends only on the entries of the submatrices $(\beta_{ij})_{i,j\in I_h}$, $h=1,\ldots,q$, whereas the right-hand side only depends on the other entries. So, if the former are large enough with respect to the absolute values of the latter, $(B_3)$ is satisfied. For example, if we take $\ell=2q$ and the matrix is
\begin{align*}
\begin{pmatrix}
\lambda & \lambda & \beta_{13} & \beta_{14} &\beta_{15} & \ldots & \beta_{1\ell}\\
\lambda & \lambda &  \beta_{23} & \beta_{24} &\beta_{25} & \ldots & \beta_{2\ell}\\
\beta_{31} & \beta_{32} & \lambda & \lambda & \beta_{35} & \ldots & \beta_{3\ell}\\
\beta_{41} & \beta_{42} & \lambda & \lambda & \beta_{45} & \ldots & \beta_{4\ell}\\
\vdots & \vdots &  &  & \ddots &  & \vdots\\
\beta_{\ell-1\,1} &  &  & \ldots & \beta_{\ell-1\,\ell-2} & \lambda & \lambda\\
\beta_{\ell 1} &  & & \ldots & \beta_{\ell\, \ell-2} & \lambda & \lambda\\
\end{pmatrix}.
\end{align*}
with $\lambda>0$ and $\beta_{ji}=\beta_{ij}<0$, then $(B_1)$ and $(B_2)$ are satisfied. If, additionally,
\begin{align*}
\lambda> 4^{\frac{2p-1}{p-1}}(q-1)C_*\qquad\text{and}\qquad |\beta_{ij}|\leq 1,
\end{align*}
then, for any $h=1,\ldots,q$, 
\begin{align*}
\Big(\min_{\{i,j\}\in E_h}\beta_{ij}\Big)\left[\frac{\min\limits_{h=1,\ldots,q} \ \max\limits_{i\in I_h}\beta_{ii}}{\dsum_{i,j\in I_h}\beta_{ij}}\right]^\frac{p}{p-1}=\lambda\left[\frac{\lambda}{4\lambda}\right]^\frac{p}{p-1}
 >\,C_*4(q-1)\geq\,C_*\dsum_{\substack{k=1 \\ k\neq h}}^q \ \dsum_{\substack{i\in I_h \\ j\in I_k}}|\beta_{ij}|   
 \end{align*}
so $(B_3)$ is satisfied.

\end{itemize}
\end{examples}

\appendix

\section{An auxiliary result}

\begin{lemma} \label{independent}
For every $r\geq 1$ there is a linear operator $E_r: H^{1}(\mathbb{R}^{N}\smallsetminus B_{r})\rightarrow H^{1}(\mathbb{R}^{N})$ such that, for every $u\in H^{1}(\mathbb{R}^{N}\smallsetminus B_{r})$,
\begin{itemize}
\item[$(i)$] $E_ru=u$ a.e. in $\mathbb{R}^{N}\smallsetminus B_{r}$,
\item[$(ii)$] $|E_ru|_2^2\leq C_1|u|_{L^2(\rn\smallsetminus B_r)}^2$
\item[$(iii)$] $\|E_ru\|^2\leq C_1\|u\|_{H^1(\rn\smallsetminus B_r)}^2$
\end{itemize}
for some positive constant $C_1$ depending only on $N$ and not on $r$. As a consequence, given $p\in(1,\frac{2^*}{2})$ there is a positive constant $C$ depending only on $N$ and $p$ such that
$$|u|_{L^{2p}(\rn\smallsetminus B_r)}\leq C\|u\|_{H^1(\rn\smallsetminus B_r)}\quad\text{for every \ }u\in H^{1}(\mathbb{R}^{N}\smallsetminus B_{r})\text{\ and every \ }r\geq 1.$$
\end{lemma}

\begin{proof}
Fix a linear (extension) operator $E_1: H^{1}(\mathbb{R}^{N}\smallsetminus B_{1})\rightarrow H^{1}(\mathbb{R}^{N})$ and a positive constant $C_1$ satisfying $(i)$, $(ii)$ and $(iii)$ for $r=1$; see e.g. \cite[Theorem 2.3.2]{kesavan}. For $r>1$, set $\what u(x):=u(rx)$ and, for $u\in H^1(\rn\smallsetminus B_r)$, define 
$$(E_ru)(y):=(E_1\what u)\Big(\frac{y}{r}\Big).$$
Then, $\widehat{E_ru}=E_1\what u$. Clearly, $E_r$ satisfies $(i)$. Note that $|\what u|_{L^2(\rn\smallsetminus B_1)}^2=r^{-N}|u|_{L^2(\rn\smallsetminus B_r)}^2$ and that
$$\|\what u\|_{H^1(\rn\smallsetminus B_1)}^2=r^{-N}\left(\int_{\rn\smallsetminus B_r}\Big(r^2|\nabla u|^2+|u|^2\Big)\right).$$
Similar identities hold true when we replace $\rn\smallsetminus B_1$ and $\rn\smallsetminus B_r$ with $\rn$. Therefore,
\begin{align*}
r^{-N}|E_ru|_2^2=|\widehat{E_ru}|_2^2=|E_1\what u|_2^2 \leq C_1\|\what u\|_{L^2(\rn\smallsetminus B_1)}^2=r^{-N}C_1|u|_{L^2(\rn\smallsetminus B_r)}^2,
\end{align*}
which yields $(ii)$. Furthermore,
\begin{align*}
r^{-N}\left(\irn\Big(r^2|\nabla(E_ru)|^2+|E_ru|^2\Big)\right)&=\|\widehat{E_ru}\|^2=\|E_1\what u\|^2 \\
&\leq C_1\|\what u\|_{H^1(\rn\smallsetminus B_1)}^2=r^{-N}C_1\left(\int_{\rn\smallsetminus B_r}\Big(r^2|\nabla u|^2+|u|^2\Big)\right).
\end{align*}
This inequality, combined with $(ii)$, yields
\begin{align*}
r^2\|E_ru\|^2&=\irn\Big(r^2|\nabla(E_ru)|^2+|E_ru|^2\Big)+(r^2-1)\irn|E_ru|^2 \\
&\leq C_1\int_{\rn\smallsetminus B_r}\Big(r^2|\nabla u|^2+|u|^2\Big)+C_1(r^2-1)\int_{\rn\smallsetminus B_r}|u|^2=r^2 C_1\|u\|_{H^1(\rn\smallsetminus B_r)}^2,
\end{align*}
which gives $(iii)$.

For $p\in(1,\frac{N}{N-2})$ let $C_2=C_2(N,p)$ be the constant for the Sobolev embedding $H^1(\rn)\subset L^{2p}(\rn)$. Then, for any $u\in H^1(\rn\smallsetminus B_r)$, using statements $(i)$ and $(iii)$ we obtain
$$|u|_{L^{2p}(\rn\smallsetminus B_r)}^2\leq |E_ru|_{2p}^2\leq C_2\|E_ru\|^2 \leq C_2C_1\|u\|_{H^1(\rn\smallsetminus B_r)}^2,$$
as claimed.
\end{proof}

\bigskip

\end{document}